\documentclass[12pt]{amsart}
\usepackage{amsmath}
\usepackage{amsthm}
\usepackage{amssymb}
\usepackage{amscd}
\usepackage{epsfig}
\usepackage{xypic}

\newtheorem{thm}{Theorem}[section]
\newtheorem{lem}[thm]{Lemma}
\newtheorem{prop}[thm]{Proposition}
\newtheorem{defn}[thm]{Definition}
\newtheorem{conj}{Conjecture}

\theoremstyle{remark}
\newtheorem{rem}[thm]{Remark}

\def\a{\mathbb A}\def\ad{{\a^\infty}}
\def\c{\mathbb C}\def\f{\mathbb F}\def\g{\mathbb G}
\def\n{\mathbb N}\def\q{\mathbb Q}\def\r{\mathbb R}
\def\z{\mathbb Z}

\def\A{\mathcal A}\def\B{\mathcal B}\def\C{{\mathcal C}}
\def\D{\mathcal D}\def\G{\mathcal G}\def\L{\mathcal L}
\def\M{\mathcal M}\def\K{\mathcal K}\def\S{\mathcal S}
\def\T{\mathcal T}\def\X{{\mathcal X}}\def\Y{{\mathcal Y}}
\def\Z{\mathcal Z}\def\H{\mathcal H}\def\O{\mathcal O}

\def\gg{\mathfrak{g}}\def\gh{\mathfrak{h}}\def\gp{\mathfrak{p}}
\def\gP{\mathfrak{P}}\def\gR{\mathfrak{R}}\def\gS{\mathfrak{S}}
\def\gt{\mathfrak{t}}\def\gx{\mathfrak{x}}

\def\Abar{\overline A}\def\Bbar{\overline B}\def\Cbar{\overline C}
\def\Dbar{\overline D}\def\Kbar{\overline K}\def\Mbar{\overline M}
\def\Sbar{\overline S}\def\Tbar{\overline T}\def\Xbar{\overline X}
\def\Ybar{\overline Y}\def\Zbar{\overline Z}\def\xbar{{\overline x}}
\def\ibar{\overline\iota}

\def\Gd{\mathaccent94G}\def\Bd{\mathaccent94B}
\def\Td{\mathaccent94T}\def\md{\mathaccent94\mu}\def\chd{\mathaccent94\chi}
\def\phd{\mathaccent94\phi}\def\zd{\mathaccent94\z}

\def\Spec{\operatorname{Spec}}\def\Lie{\operatorname{Lie}}
\def\Hom{\operatorname{Hom}}\def\Gal{\operatorname{Gal}}
\def\End{\operatorname{End}}\def\Aut{\operatorname{Aut}}

\def\dim{\operatorname{dim}}\def\codim{\operatorname{codim}}
\def\tr{\operatorname{tr}}\def\trdeg{\operatorname{trdeg}}

\title{The congruence relation in the non-$PEL$ case}
\author{Oliver B\"ultel}

\address{Mathematisches Institut der Universit\"at Heidelberg\\
Im Neuenheimer Feld 288\\D-69120 Heidelberg\\Germany}
\email{bueltel@mathi.uni-heidelberg.de}

\thanks{MSC(2000): 11G18, 14G35, 11F55, 14K22, supported by EPSRC grant}

\begin{document}
\maketitle

\begin{abstract}
This work settles the Eichler-Shimura congruence relation of Blasius and Rogawski for certain $5$-dimensional Hodge-type Shimura varieties, that were 
not tractable by previously known methods. In a more general context we introduce a hypothesis called $(NVC)$ on the behavior of Hecke correspondences, 
and show that it implies the congruence relation. A major ingredient in the proof of this result is a theorem of R.Noot on $CM$-lifts of ordinary points in 
characteristic $p$, along with an analysis of the $\pmod p$-reductions of various Hecke translates of that $CM$-lift. Finally we prove this $(NVC)$-hypothesis 
for our particular Shimura $5$-folds, and in doing so we obtain an unconditional result for the congruence relation of these non-$PEL$ examples.
\end{abstract}

\tableofcontents

\section{Introduction}
Let $G$ be a connected reductive group over $\q$. Let $X$ be a $G(\r)$-conjugacy class of algebraic homomorphisms 
$h:\c^\times\rightarrow G(\r)$. According to Deligne one calls $(G,X)$ a Shimura datum if the following axioms are fulfilled:\\
\begin{itemize}
\item[(S1)]
The Hodge structure on $\Lie G\times\r$ which is determined by $h\in X$ is of type $\{(-1,1),(0,0),(1,-1)\}$.
\item[(S2)]
The adjoint action of $h(\sqrt{-1})$ on $G^{ad}\times\r$ is a Cartan involution.
\item[(S3)]
$G^{ad}$ has no non-trivial $\q$-factors whose real points are compact.\\
\end{itemize}
Let $\ad=\q\otimes\zd$ be the ring of finite adeles and let $K$ be a compact open subgroup of 
$G(\ad)$. Then one attaches to the Shimura datum $(G,X)$ corresponding Shimura varieties
\begin{equation}
\label{quotient}
{_K}M_\c(G,X)=G(\q)\backslash(G(\ad)/K\times X).
\end{equation}
These varieties are known to have canonical models ${_K}M(G,X)$ over the reflex field $E$, see body of text for unexplained notions. Write $K$ as a product 
$K_R\times\prod_{p\notin R}K_p$ for a suitable finite set $R$ of primes, hyperspecial subgroups $K_p\subset G(\q_p)$ and a compact open subgroup 
$K_R\subset\prod_{p\in R}G(\q_p)$. According to ~\cite{blasius} one attaches a Hecke polynomial $H_\gp\in\H(G(\q_p)//K_p,\q)[t]$ to each prime $\gp$ of $E$, 
lying over $p\notin R$, here $\H(G(\q_p)//K_p,\q)$ denotes the convolution algebra of $\q$-linear combinations of $K_p$-double cosets of $G(\q_p)$. This algebra 
acts on the $\ell$-adic cohomology of ${_K}M(G,X)\times_E\q^{algcl}$ as does the absolute Galois group of $E$. The following conjecture is familiar:
\begin{conj}[Congruence Relation]
\label{CR}
The $\Gal(\q^{algcl}/E)$-module
$$V=\bigoplus_iH_{et}^i({_K}M(G,X)\times_E\q^{algcl},\q_\ell)$$
is unramified at all primes $\gp|p$ with $p\notin R$, and one has
$$H_\gp(Frob_\gp|_V)=0$$
here $H_\gp\in\H(G(\q_p)//K_p,\q)[t]$ is the $\gp$th Hecke polynomial.
\end{conj}
The conjecture is suggested by Langlands' philosophy on global $L$-parameters, see  ~\cite{blasius} for more background information. For the groups $GSp_4$ and $GL_2$ 
the congruence relation has been studied by Shimura, who also observed that in the special case of $G=GL_2$ conjecture ~\ref{CR} can be used to completely determine 
the characteristic polynomials of the elements $Frob_\gp$, acting on $V$, and on the subspaces cut out by newforms (cf. ~\cite[paragraph 6]{casselman}, ~\cite{shimura}).\\
More generally, the congruence relation is known to be true for the Siegel modular variety. The proof, which is explained in ~\cite[chapter 7]{faltings} uses in an essential 
way that ${_K}M(GSp_{2g},\gh_g^\pm)$ has a moduli interpretation in terms of principally polarized abelian varieties of dimension $g$ with level $K$-structure, which one 
can use to exhibit an integral model $\M/\z$. At least at almost all prime numbers $p$ the $\ell$-adic cohomology of ${_K}M(GSp_{2g},\gh_g^\pm)$ will be identical to the 
one of the special fiber $\M\times\f_p$ which allows to deduce the congruence relation from a correspondence theoretic reformulation involving the geometric Frobenius correspondence together with the mod-$p$ reductions of the Hecke correspondences (cf. ~\cite[chapter 7, theorem 4.2]{faltings}). The proof of this, finally has two parts:
\begin{itemize}
\item[(A)]
The computational part of the proof shows that certain Hecke translated points coincide in characteristic $p$. This result uses deformations 
of ordinary abelian varieties in terms of canonical coordinates of Serre and Tate, cf. ~\cite[chapter 7, proposition 4.1.]{faltings}
\item[(B)]
The more conceptual part of the proof needs to establish a density result for $p$-isogenies between ordinary abelian varieties. This is dealt 
with by modifying the Norman/Oort technique of displayed Cartier modules of formal groups, cf. ~\cite[chapter 7, proposition 4.3.]{faltings}.
\end{itemize}
Observe that assertion (B) is stronger than merely saying that the ordinary locus of Siegel space is dense in the special fiber over some prime $p$ of good reduction, which is 
an earlier result of N.Koblitz. The object of this work finally is threefold:\\
Firstly we generalize statement (A) to arbitrary Hodge type Shimura varieties ${_K}M(G,X)/E$ by looking at the behavior of $CM$-points with ordinary reduction, just as Shimura did this in the $GL_2$-case (~\cite[theorem 7.9]{shimura}), here we fix an imbedding of the Shimura variety into Siegel space, and therefore we get an integral model $\M$ by taking the schematic image.\\
If a modular interpretation of $\M$ is not available it is unclear how to state something to the effect of (B), but we secondly describe a 
somewhat weaker condition on the generic finiteness of the Hecke-correspondences, and assuming its validity we prove the correspondence 
theoretic congruence relation. Results in ~\cite{noot} enter into this in an essential way. Here we always assume that the group $G$ is an inner 
form of a split one, a condition that is vacuous in the Siegel case. Consequently the reflex field $E$ of our Shimura variety must be $\q$. We 
assume also that $\M\times\z_{(p)}$ is well behaved in certain senses which forces us to confine our results to 'all but finitely many $p$' only.\\
Thirdly we apply these methods to $G=GSpin(5,2)$, our main result being here that the Shimura varieties to groups of this form do satisfy this condition that serves as our 
analog of (B). The idea of proof however is different from Chai-Faltings' proof of (B). Rather than using deformation theory of formal groups, which in the case at hand is not available, we consider a slightly bigger Shimura variety which is $PEL$, and we argue by using the fiber dimensions of its Hecke-correspondences, which were computed 
in ~\cite{bu3}. It turns out that these fiber dimensions are small enough to imply the generic finiteness for Hecke correspondences of both the original Shimura variety, 
as well as the slightly bigger $PEL$-one. For these two families of examples we therefore obtain a proof of conjecture ~\ref{CR}, for all but finitely many primes $p$.\\
With the exception of the two appendices, this work is based on results proved in the author's 1997 Oxford thesis ~\cite[theorem 2.4.5/theorem 3.4.2]{bu1}. It is a great 
pleasure to thank my advisor Prof.R.Taylor, further thanks go to Prof.D.Blasius, Eike Lau, Prof.B.Moonen, Prof.R.Noot and Prof.T.Wedhorn.

\section{Operations on Cosets}
In this section we introduce modules $\gS$ and $\H$ of $R$-valued functions on $p$-adic groups, here $R$ is some commutative 
$\q$-algebra. We introduce the operations $*$, $\S$, and $|$, and explain how to interpret them in terms of ($R$-linear combinations 
of) cosets. Throughout all of the section we fix a local non-archimedean field $k$ of characteristic $0$. We write $q$ for the order of its 
residue field and we choose once and for all a uniformizer $\pi$. Finally let $|\cdot|$ be the valuation of $k$, normalized by $|\pi|=q^{-1}$.

\subsection{Hecke algebras}
Let $G$ be a connected linear algebraic group over $k$ and let $K$ be a compact open subgroup of the group of $k$-valued points of $G$. Let $\gS(G(k),K;R)$ 
be the space of locally constant compactly supported $R$-valued functions on $G(k)$. It is well-known that there exists a unique left invariant functional
$$\gS(G(k),K;R)\rightarrow R;f\mapsto\int_Gfdg$$
such that
$$\int_G1_Kdg=1$$
here is, for $A$ any set, $1_A$ the function giving value $1$ on $A$ and $0$ everywhere else. We make 
$\gS(G(k),K;R)$ into an algebra without unit by defining the product of say $\phi$ and $\psi$ to be the convolution
$$(\phi*\psi)(b)=\int_G\phi(a)\psi(a^{-1}b)da,$$
of which the associativity is easily checked. It will be useful to know the effect of $*$ 
on functions of the form $1_A$. For convenience write $|A|$ instead of $\int_G1_Adg$.

\begin{lem}
\label{help}
If the assumptions are as above then one has:
\begin{itemize}
\item[(i)]
$1_{g_1K}*1_{g_2K}=|K\cap g_2Kg_2^{-1}|1_{g_1Kg_2K}$
\item[(ii)]
$1_{g_1K}*1_{Kg_2K}=1_{g_1Kg_2K}$
\end{itemize}
\end{lem}
\begin{proof}
It is obvious that the left hand side of (i) can only be non-zero on the set $g_1Kg_2K$. Consider any element in this set and write it as $b=g_1k_1g_2k_2$, then
$$(1_{g_1K}*1_{g_2K})(b)=\int_G1_{g_1K}(a)1_{g_2K}(a^{-1}b)da$$
$$=|g_1K\cap bKg_2^{-1}|=|(g_1k_1)K\cap (g_1k_1)g_2k_2Kg_2^{-1}|=|K\cap g_2Kg_2^{-1}|$$
because of the left invariance of $|\dots|$. Property (ii) follows from it because
$$1_{g_1K}*1_{Kg_2K}=|K\cap g_2Kg_2^{-1}|^{-1}1_{g_1K}*1_K*1_{g_2K}$$
$$=|K\cap g_2Kg_2^{-1}|^{-1}1_{g_1K}*1_{g_2K}=1_{g_1Kg_2K}$$
\end{proof}

\begin{defn}
Let $G$, and $K$ be as in the beginning of the subsection. Then put
\begin{eqnarray*}
&&\H(G(k)//K,R)=1_K*\gS(G(k),K;R)*1_K\\
&&\H(G(k)/K,R)=\gS(G(k),K;R)*1_K.
\end{eqnarray*}
$\H(G(k)//K,R)$ is an associative $R$-algebra with two sided identity element, called the Hecke algebra of the pair $(G,K)$.
$\H(G(k)/K,R)$ is a unital right module under this algebra, containing $\H(G(k)//K,R)$ as a right submodule.
\end{defn}

Notice, that the last lemma enables one to identify the convolution of elements in $\H(G(k)//K,R)$ with the products of cosets which is classically used to define the Hecke 
algebra. More precisely let $gK$ be a left coset and $KhK$ be a double coset, write it as finite disjoint sum of left cosets $\bigcup_ih_iK$, and observe the following:
$$1_{gK}*1_{KhK}=\sum_i1_{gh_iK}$$
Similarly if $KgK$ is a double coset with left coset decomposition $\bigcup_jg_jK$, then
$$1_{KgK}*1_{KhK}=\sum_{i,j}1_{g_jh_iK}.$$

\subsection{Reduction modulo the radical}
Let $U$ be the unipotent radical of $G$, let $M$ be the quotient of $G$ by $U$ and let $L$ be the image of $K$ in $M(k)$. Let $dm$ ($du$) 
be left-invariant measures on $M$ (on $U$) such that $L$ ($K\cap U(k)$) gets measure $1$. We want to describe a certain homomorphism
$$\S:\gS(G(k),K;R)\rightarrow \gS(M(k),L;R)$$
By choosing a Levi section we may regard $M$ as a subgroup of $G$. It is a fact that $dg=dmdu$ i.e.
$$\int_G\phi(g)dg=\int_M\int_U\phi(mu)dmdu$$
for all $\phi\in\gS(G(k),K;R)$. Here one has to exercise care with in which order to put $M$ and $U$. For this and for all the 
integration formulas used in the sequel we refer to ~\cite[IV.4.1]{cartier}. Now, if we give ourself a $\phi\in\gS(G(k),K;R)$, we define
$$\S\phi(m)=\int_U\phi(mu)du$$
Let us check that $\S$ is a homomorphism of algebras:
\begin{eqnarray*}
\S(\phi*\psi)(m)&=&\int_U\int_G\phi(a)\psi(a^{-1}mu)dadu\\
&=&\int_U\int_M\int_U\phi(nv)\psi(v^{-1}n^{-1}mu)dndvdu\\
&=&\int_U\int_M\phi(nv)(\int_U\psi(n^{-1}mu)du)dndv\\
&=&\int_M(\int_U\phi(nv)dv)(\int_U\psi(n^{-1}mu)du)dn\\
&=&\S\phi*\S\psi(m)
\end{eqnarray*}
The definition of $\S$ does not depend on the choice of a Levi section. Note also that $\S1_K=1_L$, so that $\S$ induces a homomorphism
from $\H(G(k)//K,R)$ to $\H(M(k)//L,R)$. In order to compute the effect of $\S$ on cosets we need another lemma:

\begin{lem}
\label{effect}
Let $g\in G(k)$, then:
$$\S1_{gK}=1_{gL}$$
\end{lem}
\begin{proof}
The left hand side vanishes off $gL$, so consider an element $m\in gL$, i.e. a $m\in M(k)$ 
such that there exists a $u_0\in U$ with $mu_0\in gK$, say $mu_0=gk$. Then we can compute:
$$\S1_{gK}(m)=\int_U1_{gK}(mu)du=\int_U1_{gK}(mu_0u)du=$$
$$\int_U1_{gK}(gku)du=|K\cap U|=1$$
\end{proof}

\subsection{Restriction to parabolic subgroups}
Assume that $G$ is reductive and unramified. Assume also that $K$ is a hyperspecial subgroup of $G(k)$ and let finally $P$ be any 
parabolic subgroup of $G$. A function $\phi$ on $G(k)$ may be restricted to $P(k)$. It turns out that this sets up an algebra homomorphism
$$|_P:\H(G(k)//K,R)\rightarrow\H(P(k)//K\cap P(k),R)$$
and a $|_P$-linear module homomorphism
$$|_P:\H(G(k)/K,R)\rightarrow\H(P(k)/K\cap P(k),R)$$
Let us prove this. The left invariant measure of $P$, normalized in order to give $K\cap P(k)$ measure $1$ shall be denoted 
by $dp$. By the Iwasawa decomposition one has $G(k)=P(k)K$, further it is a fact that $dg=dpdk$, which means that
$$\int_G\phi(g)dg=\int_P\int_K\phi(pk)dpdk$$
for all locally constant compactly supported $\phi$. Now consider $\phi\in\H(G(k)/K,R)$, $\psi\in\H(G(k)//K,R)$, let $p_0$ be in $P(k)$ and compute:
\begin{eqnarray*}
(\phi*\psi)(p_0)&=&\int_P\int_K\phi(pk)\psi((pk)^{-1}p_0)dpdk\\
&=&\int_P\phi(p)\psi(p^{-1}p_0)dp\\
&=&(\phi|_P*\psi|_P)(p_0)
\end{eqnarray*}
Again it is useful to derive what this means for cosets. Let $gK$ be a coset, assume that 
$g\in P(k)$. Obviously $1_{gK}|_P=1_{g(K\cap P(k))}$ then, and similarly for double cosets.

\section{A Proposition on the Hecke Polynomial}
\label{Hecke}
In this section we review the Satake isomorphism of a split reductive group $G$ over $k$ first. We fix $T$, a split maximal torus, and $B$, a Borel group 
containing $T$, throughout. We refer to them as a Borel pair of $G$. Consider the corresponding based root datum $\psi(G)=(X^*(T),X_*(T),\Delta^*,\Delta_*)$ 
(characters, cocharacters, simple roots, simple coroots), recall that the unique reductive $\c$-group $\Gd$ with root datum $(X_*(T),X^*(T),\Delta_*,\Delta^*)$ 
say with respect to a Borel pair $(\Bd,\Td)$ is called the dual of $G$. Let $\delta$ be the character of $T$ which arises by adding all positive roots.\\
Finally we recall well-known properties of the Hecke polynomial $H_{G,\mu}$, which we introduce 
in this local setting according to ~\cite{blasius}. Here $\mu$ is a minuscule cocharacter of $G$.

\subsection{The formalism of the dual group}

Let $\Omega\subset\Aut(T)$ be the Weyl group of $G$. Its significance to Hecke algebras stems from the following:

\begin{thm}[Satake]
\label{satake}
Let $T\subset B\subset G$ be as above. Let $K$ and $T_c$ be hyperspecial subgroups of $G(k)$ and $T(k)$. The composition of the three maps as shown below
$$\begin{CD}
\H(G(k)//K,\c)@>\mathaccent94{}>>\H(T(k)//T_c,\c)\\
@V|_BVV@A\cdot|\delta|^{1/2}AA\\
\H(B(k)//K\cap B(k),\c)@>\S>>\H(T(k)//T_c,\c)
\end{CD}$$
is independent of the choice of $B$. Its image consists of the $\Omega$-invariants of $\H(T(k)//T_c,\c)$ and one thus obtains an isomorphism
$$\phi\mapsto \phd$$
$$\H(G(k)//K,\c)\rightarrow \H(T(k)//T_c,\c)^\Omega$$
\end{thm}
\begin{proof}
This is shown in ~\cite{satake1}.
\end{proof}

This function $\phd$ on $T(k)$ is called the Satake transform of $\phi$. It is useful to spell out the 
following functoriality property of the Satake transform. The proof is trivial and will be omitted.

\begin{lem}
Let $T$, $G$, and $K$ be as above and let $P$ be a parabolic subgroup which contains $T$. Let $U$ be the unipotent 
radical of $P$, and let $\delta_P$ be the character defined by $\det(ad(t)|_{\Lie U})$. Then the Levi quotient $M=P/U$ is 
split reductive, $L=K\cap P(k)/K\cap U(k)$ is a hyperspecial subgroup of $M(k)$ and one obtains a commutative diagram,
$$\xymatrix{{\H(G(k)//K,\c)} \ar[r]^{\S\circ|_P} \ar[d]^{\mathaccent94 {}} &
{\H(M(k)//L,\c)} \ar[r]^{\cdot|\delta_P|^{1/2}} &
{\H(M(k)//L,\c)} \ar[d]^{\mathaccent94 {}} \\
{\H(T(k)//T_c,\c)^\Omega} \ar[rr] \hole && {\H(T(k)//T_c,\c)^{\Omega_M}}}$$
in which $\Omega_M\subset\Aut(T)$ is the Weyl group of $M$, and the lowest arrow is the inclusion of the $\Omega$-invariants into the $\Omega_M$-invariants.
\end{lem}

We want to bring the dual group into play whose class functions serve as a source of $\Omega$-invariant elements of $\H(T(k)//T_c,\c)$. Identify 
$\H(T(k)//T_c,\c)$ with $\c[X_*(T)]$ by sending a cocharacter $\chi:\g_m\rightarrow T$ to $1_{\chi(\pi^{-1})T_c}\in\H(T(k)//T_c,\c)$. A class function on 
$\Gd$ restricts to a $\Omega$-invariant function on $\Td$. By the very definition of the dual torus the ring of algebraic functions on $\Td$ is our $\c[X_*(T)]$. 
Finally by a classical argument of Chevalley the subalgebra of its $\Omega$-invariants consists precisely of functions which arise from class functions 
on $\Gd$. For example let $\chi:\g_m\rightarrow G$ be a central cocharacter. We have a corresponding $\chd:\Gd\rightarrow\g_m$. Under the above 
considerations the function $\Gd\rightarrow\c$, $g\mapsto\chd(g)$ corresponds to the element $\chi(\pi^{-1}):=1_{\chi(\pi^{-1})K}\in\H(G(k)//K,\c)$.

\subsection{The Hecke polynomial}
Consider a conjugacy class of minuscule cocharacters $\mu:\g_m\rightarrow G$. If $T\subset B\subset G$ is as before, 
we may consider the unique representative of that conjugacy class which has image in $T$ and which is dominant 
relative to $B$, that means that $<\mu,\alpha>\geq0$ for all positive roots $\alpha$. We set $d=<\mu,\delta>$.\\
Write $\md$ for the corresponding character of $\Td$. Let $r:\Gd\rightarrow GL(V)$ be the irreducible 
representation of $\Gd$ with highest weight $\md$ relative to $\Bd$ and $\Td$, as in ~\cite{blasius}.
\begin{defn}

Let $G/k$ be split reductive and $\mu$ be a conjugacy class of minuscule cocharacters of $G$. Consider the associated representation
$$r:\Gd\rightarrow GL(V)$$
Then $\det_V(t-q^{d/2}r(g))$ is a monic polynomial in $t$ whose coefficients are class functions on $\Gd$. It thus defines a polynomial
$$H_{G,\mu}\in\H(G(k)//K,\c)[t],$$
which is called the Hecke polynomial.
\end{defn}

The following fact describes a zero of $H_{G,\mu}$ in a larger Hecke algebra:

\begin{prop}
\label{trivial}
Let $G/k$ be split reductive with a hyperspecial subgroup $K$. Let $\mu$ be a minuscule cocharacter of $G$. Let $P$ be the parabolic subgroup which is determined by 
$\mu$, i.e. the biggest parabolic subgroup of $G$ relative to which $\mu$ is dominant. Let $M$ be the Levi quotient of $P$ and let $L$ be the image of $K\cap P(k)$ in 
$M(k)$. Let $H_{G,\mu}\in\H(G(k)//K,\c)[t]$ be the Hecke polynomial. Then
$$H_{G,\mu}(\mu(\pi^{-1}))=0$$
here we regard the equation taking place in $\H(M(k)//L,\c)$ via the natural map $\S\circ|_P:\H(G(k)//K,\c)\rightarrow\H(M(k)//L,\c)$.
\end{prop}
\begin{proof}
Let $U$ be the unipotent radical of $P$. Let $\delta_P$ be the character defined by $t\mapsto\det(ad(t)|_{\Lie U})$. Choose a Borel pair $T\subset B\subset P$. The 
Satake isomorphism of $G$ factors through the map $\S\circ|_P$, so that we may check the requested identity in $\H(T(k)//T_c,\c)$. We first check that upon applying
$$\begin{CD}
\H(M(k)//L,\c)@>\cdot|\delta_P|^{1/2}>>\H(M(k)//L,\c)
\end{CD}$$
to the element $\mu(\pi^{-1})$ we get $q^{d/2}\mu(\pi^{-1})$. To see this write $\gg=\gg_{-1}\oplus\gg_0\oplus\gg_1$ for the eigenspace decomposition of 
the adjoint action of $\mu$ on $\Lie G=\gg$, indexed by $\{-1,0,1\}$ as $\mu$ is minuscule. Due to $\Lie P=\gg_0\oplus\gg_1$ and $\Lie U=\gg_1$, we have
$$<\delta_P,\mu>=\dim\gg_1.$$
Now write $\gg=\gt\oplus\bigoplus_{\alpha\in\Phi}\gg_\alpha$ for the root space decomposition of 
$\gg$. As $\mu$ is dominant we have that $<\alpha,\mu>=1$ implies $\alpha\in\Phi^+$, therefore
$$<\delta,\mu>=\sum_{\alpha\in\Phi^+}<\alpha,\mu>=|\{\alpha\in\Phi|<\alpha,\mu>=1\}|=\dim\gg_1,$$
and hence
$$|\delta_P(\mu(\pi^{-1}))|^{1/2}=|\pi^{-1}|^{d/2}=q^{d/2}.$$
We have to show that the Satake transform of the element $q^{d/2}\mu(\pi^{-1})\in\H(M(k)//L,\c)$ is annihilated by $H_{G,\mu}$. Take 
into account that the cocharacter $\mu$ gives rise to a character $\md$ of $\Td$. The element $\mu(\pi^{-1})$ corresponds to the function 
$g\mapsto\md(g)$. The representation $V$ with highest weight $\md$ has a weight space decomposition $V=\bigoplus_\lambda V_\lambda$. Let 
$0\neq x\in V_{\md}$, so that one has $r(g)x=\md(g)x$, but then $\det_V(q^{d/2}\md(g)-q^{d/2}r(g))=0$, for all $g\in\Td$ which is what we wanted.
\end{proof}

\begin{rem}
It is a fact that $H_{G,\mu}$ is actually in $\H(G(k)//K,\q)[t]$.
\end{rem}

\section{A Proposition on $CM$-points}
\label{cm}
We write $\q^{algcl}$ for the algebraic closure of $\q$ in $\c$. In a sense which will be made precise later we prove that certain Hecke operators 
act trivially on the set of mod $\gP$-reductions of ordinary $CM$-points over $\q^{algcl}$. Our exposition of this is inspired by ~\cite{deligne}.\\
Then we move on to derive the correspondence theoretic congruence relation. We do that by appealing to a theorem of R.Noot on the existence of $CM$-lifts for ordinary 
closed points. Hence, we may apply the theory of complex multiplication to such a $CM$ lift, as the action of the Frobenius and of Hecke correspondences are controlled by 
the $CM$-type. To deduce the congruence relation one still has to assume that the correspondences in question have 'no vertical components' (see Definition ~\ref{nvc}).

\subsection{$CM$-points}
\label{point}
An abelian variety $A/\q^{algcl}$ is said to be of $CM$-type if one (hence every) maximal commutative semisimple subalgebra of 
$\End^0(A)$ has degree equal to $2\dim A$. Consider the integral and rational Tate modules
$$T_pA=\lim_\leftarrow A[p^n],$$
$$V_pA=\q\otimes T_pA$$
of $A$. If $A$ has complex multiplication by say $L\subset\End^0(A)$, we also have a $L_\gp$-vector space $V_\gp A$ for every prime $\gp$ of $L$ above $p$. 
These are the direct summands of $V_pA$ that correspond to the idempotents in the algebra $L\otimes\q_p=\bigoplus_{\gp|p}L_\gp$, which acts on $V_pA$. 
To $(A,L)$ one attaches the $CM$-type, which is by definition the set $\Sigma$ of algebra morphisms $\sigma:L\rightarrow\q^{algcl}$ with non-zero eigenspace 
in $\Lie(A)$, recall also that one calls the subfield $F$ of $\q^{algcl}$ generated by the elements $\{\tr(a|_{\Lie(A)})|a\in L\}$ the reflex field of $\Sigma$. Let $\gP$ 
be a prime of $\q^{algcl}$ with valuation ring $\O$. Then $A$ extends to an abelian $\O$-scheme $\A$, the special fiber of which we denote by $\Abar$. Recall 
finally that the $p$-rank of $\Abar$ is equal to $\dim A$ if and only if $\gP$ induces a completely split prime in $F$, i.e. if the $\gP$-adic completion of $F$ 
is $\q_p$. We call either $A$ or $(L,\Sigma)$ ordinary at $\gP$ if this is the case. If $(L,\Sigma)$ is ordinary at $\gP$ then there exists a set $S$ of primes 
of $L$ over $p$ such that $\Sigma=\{\sigma:L\rightarrow\q^{algcl}|\sigma^{-1}(\gP)\in S\}$. As in ~\cite[section 4]{deligne} we define refined Tate modules:

\begin{defn}
Let $L$ be a $CM$-algebra/$\q$ and let $\Sigma$ a be $CM$-type for $L$. Let $A/\q^{algcl}$ have $CM$ by $(L,\Sigma)$, assumed to be ordinary at $\gP$. Then we set:
\begin{eqnarray*}
V_p''A&=&\bigoplus_{\gp\in S} V_\gp A\\
V_p'A&=&\bigoplus_{\gp\notin S} V_\gp A\\
T_p''A&=&T_pA\cap V_p''A\\
T_p'A&=&(T_pA+V_p''A)\cap V_p'A,
\end{eqnarray*}
here $\gp$ always denotes a prime of $L$ over $p$, and $S$ is the set of $\gp|p$ with $\Sigma=\{\sigma|\sigma^{-1}(\gP)\in S\}$.
\end{defn}

These refined Tate modules make it possible to describe congruences between $CM$-abelian varieties:

\begin{lem}
\label{heuristic}
Let $A$ and $B$ be two abelian varieties over $\q^{algcl}$. Let $u\in\Hom(A,B)\otimes\z[1/p]$ be a $p$-isogeny with inverse
$v\in\Hom(B,A)\otimes\z[1/p]$. Assume $A$ and $B$ have $CM$ by $L\subset\End^0(A)=\End^0(B)$ with type $\Sigma$. Assume that
$\Sigma$ is ordinary at $\gP$. If
\begin{eqnarray}
u(T_p'A)=T_p'B&,&u(T_p''A)=T_p''B
\end{eqnarray}
then by reducing modulo $\gP$ one obtains isomorphisms ${\overline u}\in\Hom(\Abar,\Bbar)$, ${\overline v}\in\Hom(\Bbar,\Abar)$.
\end{lem}
\begin{proof}
Write $\A$ and $\B$ for the extensions of $A$ and $B$ to abelian $\O$-schemes. Let $R\subset L$ be the order which acts on $A$, let $R''$ and $R'$ 
be the images of $R\otimes\z_p$ in $\bigoplus_{\gp\in S}L_\gp$ and $\bigoplus_{\gp\notin S}L_\gp$. We assume that $R\otimes\z_p$ is a direct sum 
of $R''$ and $R'$, or equivalently, that $T_pA$ is a direct sum of $T_p'A$ and $T_p''A$. This is no loss of generality as we may replace $A$ by an 
isogenous abelian variety $A^*$ which has this property, to then apply the result twice to isogenies $A^*\rightarrow A$ and $A^*\rightarrow B$.\\
Note that $\A[p^\infty]$ decomposes into a direct sum of $p$-divisible groups $\A'$ and $\A''$, corresponding to the idempotents in $R'\oplus R''$. Let 
$\K\subset\A$ be the kernel of (the extension to $\A$ of) the isogeny $p^nu$. Let further $\K''$ be the intersection of $\K$ with $\A''$, and let $\K'=\K/\K''$ 
be the image of $\K$ in $\A'$. Let $\B''$ be the $p$-divisible group $\A''/\K''$ and $\B'$ be $\A'/\K'$. Notice that the ($p$-divisible) groups $\K''$, $\A''$, and 
$\B''$ are multiplicative whereas $\K'$, $\A'$, and $\B'$ are \'etale, by ~\cite[th\'eor\`eme 3]{hondatate}. From definition one has a short exact sequence
$$0\rightarrow\B''[p^\infty]\rightarrow\B[p^\infty]\rightarrow\B'[p^\infty]\rightarrow0$$
of which the special fiber splits canonically, consequently $\Kbar=\Kbar'\oplus\Kbar''$. Moreover, from the assumptions 
on $u$ one infers $\K'=\A'[p^n]$, and $\K''=\A''[p^n]$, it follows that $\Kbar=\Abar[p^n]$, which is what we wanted.
\end{proof}

\subsection{Shimura subvarieties}
\label{shsub}
We move on to the consequences for the $CM$-points on a Shimura variety. We fix a Shimura datum $(G,X)$. An element $h$ of $X$ is a homomorphism:
$$h:Res_{\c/\r}\g_m\rightarrow G\times\r,$$
which by base change gives rise to
$$h_\c:\g_m\times\g_m\rightarrow G\times\c,$$
the restriction of this to the first copy of $\g_m$ is a cocharacter, denoted by
$$\mu_h:\g_m\rightarrow G\times\c,$$
which is minuscule, by the general framework of Deligne, see ~\cite[1.1.1]{corvallis} for more information. Furthermore, by loc.cit one calls the reflex field 
$E\subset\c$ the field over which the conjugacy class of $\mu_h$ is defined, this field does of course not depend on the choice of $h$ within $X$.\\
One calls a specific $h\in X$ to have $CM$ if there exists a rational torus $T\subset G$ so that $h(\c^\times)\subset T(\r)$. 
If this is the case then one attaches a reciprocity law $r$ to $h$ as follows: Let $F$ be the reflex field of $(T,\{h\})$, let 
$\mu_h:\g_m\times F\rightarrow T\times F$ be the cocharacter to $h$, and let finally be $r$ the composite of the maps below:
$$\begin{CD}
{(F\otimes\a)^\times}@>{\mu_h}>>{T(F\otimes\a)}@>{\n_{F/\q}}>>{T(\a)}
\end{CD}$$
In analogy to the discussion in subsection ~\ref{point} we say:

\begin{defn}
\label{ordinary}
A $CM$-type $h\in X$ with reflex field $F$ is called ordinary at some prime $\gP$ of $\q^{algcl}$ if the prime which it induces in $F$ is completely split.
\end{defn}

If $h$ is ordinary at $\gP$ then there is an induced imbedding $F\hookrightarrow\q_p$. By base change of $\mu_h$, an ordinary $h\in X$ 
gives rise to a $\q_p$-rational cocharacter $\mu_{h,\gP}:\g_m\times\q_p\rightarrow G\times\q_p$. There is a parabolic subgroup 
$P_{h,\gP}\subset G\times\q_p$ that goes with the cocharacter $\mu_{h,\gP}$ in the usual way, being the largest group for which 
$\mu_{h,\gP}$ is dominant, we set $U_{h,\gP}$ for the unipotent radical of $P_{h,\gP}$, and $M_{h,\gP}$ for the centralizer of $\mu_{h,\gP}$.\\
We also have to fix a level structure, to this end we introduce topological rings:
\begin{eqnarray*}
&&\zd=\prod_p\z_p\\
&&\ad=\q\otimes\zd\\
&&\a=\ad\oplus\r
\end{eqnarray*}
and consider some compact open subgroup $K\subset G(\ad)$. Once a smooth connected extension of $G$ to a group $\z$-scheme 
$\G$ is fixed one can conveniently decompose $K$ into a product $K_R\times\prod_{p\notin R}K_p$, with $R$ some sufficiently big 
finite set of primes, $K_R$ some compact open subgroup of $\prod_{p\in R}G(\q_p)$, and $K_p=\G(\z_p)$ hyperspecial.\\
Now let ${_K}M(G,X)$ be a canonical model in Deligne's sense. It is an algebraic variety over $E$ with certain nice 
properties. Most notably one has that the complex points of ${_K}M(G,X)$ are parameterized by the double quotient:
$${_K}M_\c(G,X)=G(\q)\backslash(G(\ad)/K\times X),$$
i.e. for any $g\in G(\ad)$ and $h\in X$ the $G(\q)$-orbit of the pair $gK\times h$ is a point $x$ on ${_K}M_\c(G,X)$. We denote this point by $x=[gK\times h]$. 
It is called a $CM$ point if $h$ has $CM$, and the significance of those stems from them being rational over $\q^{algcl}$ and satisfying the formula:
$$\tau([gK\times h])=[r(t)gK\times h]$$
where $\tau$ is an element of the absolute Galois group of $F$ with abelianization equal to the Artin symbol of the idele $t\in(F\otimes\a)^\times$, our normalization of class field theory does follows that of Deligne (as did our discussion of the Hecke polynomial). One further bit of notation: If $x=[gK\times h]\in{_K}M_\c(G,X)$ is a $CM$ point, we denote 
by $\mu_{x,\gP}$ the cocharacter $g_p^{-1}\mu_{h,\gP}g_p$, where $g_p$ is the component at $p$ of the adelic group element $g$, it is easier to work with $\mu_{x,\gP}$ 
than with $\mu_{h,\gP}$ as it depends only on $x$ and not on a choice of representation $[gK\times h]$, the same applies to $P_{x,\gP}$, $U_{x,\gP}$, and $M_{x,\gP}$.\\
Before we state the result of this section we need to consider embeddings of ${_K}M(G,X)$ into Siegel space: For any positive integer $g$, we endow 
$\z^{2g}$ with the standard antisymmetric perfect pairing. We write $GSp_{2g}$ for the reductive group $\z$-scheme of symplectic similitudes 
of $\z^{2g}$, and $\gh_g^\pm$ for the Siegel double half space of genus $g$, $(GSp_{2g}\times\q,\gh_g^\pm)$ is a Shimura datum. From now 
on we assume that $(G,X)$ is a datum of Hodge type, which means that there exists an injection $i:G\hookrightarrow GSp_{2g}\times\q$ of algebraic 
groups such that $i$ carries the conjugacy class $X$ to the conjugacy class $\gh_g^\pm$. Recall from ~\cite[proposition 1.15.]{bourbaki}:

\begin{lem}
\label{subsh}
Let $i:(G,X)\hookrightarrow (G',X')$ be an imbedding of Shimura data. For every compact open subgroup $K\subset G(\ad)$ 
there exists a compact open $H\subset G'(\ad)$ containing $i(K)$ and small enough in order to make the induced mapping
$${_K}M_\c(G,X)\rightarrow{_H}M_\c(G',X').$$
a closed immersion.
\end{lem}

Whenever groups $K$, and $H$ with properties as in the lemma are given, we regard ${_K}M_\c(G,X)$ (resp. its canonical model ${_K}M(G,X)$ over the reflex field 
$E\subset\c$) as a Shimura subvariety of ${_H}M_\c(G',X')$ (resp. ${_H}M(G',X')\times_{E'}E$). In particular we obtain canonical models of Hodge type Shimura varieties 
${_K}M(G,X)$ as subvarieties of ${_H}M(GSp_{2g},\gh_g^\pm)\times E$, cf. ~\cite[crit\`ere 2.3.1.]{corvallis}, ~\cite[corollaire 5.4.]{bourbaki}. To make things more explicit 
observe that $H$ can be written as $H_Q\times\prod_{p\notin Q}GSp_{2g}(\z_p)$ where $Q$ is a sufficiently big finite set of primes containing $R$, and $H_Q$ is some 
compact open subgroup of $\prod_{p\in Q}GSp_{2g}(\q_p)$. Write $n$ for the product of the primes in $Q$. Observe also that the canonical model 
${_H}M(GSp_{2g},\gh_g^\pm)$ over its reflex field $\q$ has a moduli interpretation in terms of abelian varieties with level $H$-structure suggesting an extension to a scheme over $\Spec\z[\frac{1}{n}]$: Based on considerations in ~\cite[chapter 5]{kottwitz} we consider the functor giving to some base scheme $X/\z[\frac{1}{n}]$ (connected and 
equipped with a choice of geometric point $x\rightarrow X$) the set of $\z[\frac{1}{n}]$-isogeny classes of triples $(A,\lambda,\bar\eta)/X$. Here is
\begin{itemize}
\item
$A/X$ a $g$-dimensional abelian scheme,
\item
$\lambda$ a homogeneous polarization of $A/X$, containing at least one member of degree a unit in $\z[\frac{1}{n}]$,
\item
$\bar\eta$ a $\pi_1(X,x)$-invariant $H_Q$-orbit of a family of isomorphisms $\eta=(\eta_p)_{p\in Q}:\prod_{p\in Q}\q_p^{2g}\rightarrow\prod_{p\in Q}V_p(A\times_Xx)$ 
taking for each $p\in Q$ the standard antisymmetric pairing on $\q_p^{2g}$ to a multiple of the $\lambda$-induced Weil pairing on the $\q$-tensorized $p$-adic 
Tate modules of the fiber of $A/X$ over $x$.
\end{itemize}
We write $\A_{g,H}$, for the coarse moduli scheme of this functor, it exists and is quasiprojective over $\z[\frac{1}{n}]$ by ~\cite[theorem 7.9]{mumford}. 
We write $A_{g,H}$ for its generic fiber, $\A_{g,H,p}$ and $\Abar_{g,H,p}$ for stalk and fiber over any prime $p\not|n$. By removing the non-ordinary 
locus in $\Abar_{g,H,p}$ we obtain open subschemes $\A_{g,H,p}^\circ\subset\A_{g,H,p}$, and $\Abar_{g,H,p}^\circ\subset\Abar_{g,H,p}$.\\
As observed in ~\cite[lemma 2.2/remark 2.3]{bu2} a $CM$ point $x\in{_K}M_\c(G,X)$ has ordinary mod-$\gP$ reduction on 
$\A_{g,H}$ if and only if $h$ is ordinary at $\gP$ in the sense of definition ~\ref{ordinary}. With this preparation we can now state:

\begin{lem}
\label{moreconsequence}
Let $(G,X)\hookrightarrow(GSp_{2g},\gh_g^\pm)$ give rise to a Shimura subvariety $M={_K}M(G,X)\hookrightarrow A_{g,H}\times E$. 
Let $\M\hookrightarrow\A_{g,H}\times\O_E$ be the schematic closure of $M$ in $\A_{g,H}\times\O_E$. Let $\O$ be a valuation 
ring of $\q^{algcl}$ with its maximal ideal $\gP$ not containing $n$. Let $x=[gK\times h]\in M(\q^{algcl})$ be a $CM$-point with 
ordinary reduction at $\gP$. Then the cocharacter $\mu_{x,\gP}$ is well-defined, we have for all $u\in U_{x,\gP}(\q_p)$:
\begin{itemize}
\item
The $\q^{algcl}$-points $[gK\times h]$ and $[guK\times h]$ extend to $\O$-valued points of $\M$.
\item
Their special fibers $\overline{[gK\times h]}$, and $\overline{[guK\times h]}$, which are thus defined, are equal.
\end{itemize}
\end{lem}
\begin{proof}
By standard arguments of ~\cite{serretate} one finds that $CM$ points have good reduction on 
$\A_{g,H}\times\O_E$, and likewise on $\M$, cf. ~\cite[lemma 2.1]{bu2}. This shows the first assertion.\\
To do the second, observe that in terms of the moduli interpretation the point $x$ gives birth to the isogeny class of an abelian variety $A/\q^{algcl}$ with a 
homogeneous polarization $\q\lambda$ and a full level structure $\eta$. Moreover, $A$ has $CM$ say by the type $(L,\Sigma)$ with reflex $F$. The type will generate 
a decomposition $\q_p^{2g}=V'\oplus V''$, as we assumed that $\gP$ induces a completely split prime of $F$. It is not difficult to see that the $\gp$-component of the 
reciprocity law is the map $\q_p^\times\rightarrow GSp_{2g}(\q_p)$ acting trivially on $V'$ and acting with weight $1$ on $V''$, and our unipotent subgroup $U_{x,\gP}$ 
consists of all symplectic similitudes which preserve the maximal isotropic subspace $V''$. We shall look upon $U_{x,\gP}(\q_p)$ as a subgroup of $GSp_{2g}(\ad)$ 
placing $1$'s at all components different from $p$. Implicit in the identification ${_H}M(GSp_{2g},\gh_g^\pm)=A_{g,H}$ is the lattice $\z^{2g}\subset\q^{2g}$. The 
isogeny class of $(A,\q\lambda,\eta)$ contains a unique member with $\eta(\zd^{2g})=H_1(A,\zd)$. Similarly, the isogeny class of $xu=(A,\q\lambda,\eta\circ u)$ 
contains a unique member $B$ with $\eta(u(\zd^{2g}))=H_1(B,\zd)$, i.e. there exists a $\phi\in\Hom(A,B)\otimes\z[1/p]$ so that it compares $B$ to $A$ as follows:
$$\begin{CD}
{\ad^{2g}}@>>>{H_1(A,\ad)}\\
@A{u}AA@V{\phi}VV\\
{\ad^{2g}}@>>>{H_1(B,\ad)}
\end{CD}$$
The assumption on $u$ leads $\phi$ to satisfy the conditions of lemma ~\ref{heuristic}, so that we get an isomorphism
$\Abar\cong\Bbar$. Note finally that the level $H$ structure which is determined by $\eta$ will not be altered by $u$ as the latter
is placed in the $p$-component of the adele group. We can thus conclude that $\overline{[guK\times h]}=\overline{[gK\times h]}$.
\end{proof}

\subsection{Main theorem}
\label{turtleneck}
Our aim in this subsection is to construct elements in rings of correspondences $C_{fin}$, and $C_{rat}$, see appendix ~\ref{ccf} for their definition and some properties. 
We define the $(NVC)$-condition so that we can state and prove the congruence relation.\\
We keep our $(G,X)$ and retain assumptions on $K=K_R\times\prod_{p\notin R}K_p$ from subsection ~\ref{shsub}. In addition we assume that the group $G(\q)$ 
acts without fixed points on $G(\ad)/K\times X$. Level structures which fulfill this are called neat. This ensures that the Shimura variety ${_K}M_\c(G,X)$ is smooth. 
For every adelic group element $g$ we can consider the open compact subgroups $K$, $gKg^{-1}$ and $K\cap gKg^{-1}$, giving rise to Shimura varieties with 
these level structures. Multiplication by $g$ from the right gives an isomorphism from ${_{gKg^{-1}}}M(G,X)$ to ${_K}M(G,X)$ and we therefore obtain a map
$${_{K\cap gKg^{-1}}}M(G,X)\rightarrow{_K}M(G,X)\times{_K}M(G,X),x\mapsto(x,xg).$$
The push forward of the fundamental cycle of ${_{K\cap gKg^{-1}}}M(G,X)$ by this map is an element of $C_{fin}({_K}M(G,X),{_K}M(G,X))$ 
which we will denote by $[KgK]$, we write $T(g)\subset{_K}M(G,X)\times_E{_K}M(G,X)$ for the support of $[KgK]$. By $\q$-linear extension 
one gets a map $\iota:\H(G(\ad)//K,\q)\rightarrow\q\otimes C_{fin}({_K}M(G,X),{_K}M(G,X))$ from the assignment $1_{KgK}\mapsto[KgK]$. 
It is a homomorphism of algebras, we write $\iota_p$ for the restriction of $\iota$ to $\H(G(\q_p)//K_p,\q)$, here note that to any factorization 
of $K$ into $K_R\times\prod_{p\notin R}K_p$ we have a canonical factorization of $\H(G(\ad)//K,\q)$ into
$$\H(\prod_{p\in R}G(\q_p)//K_R,\q)\otimes\bigotimes_{p\notin R}\H(G(\q_p)//K_p,\q).$$
Let $\O_E$ be the ring of integers in $E$. By an integral model of ${_K}M(G,X)$ over $\O_E$ we mean a separated, smooth $\O_E$-scheme 
$\M$ that is equipped with an isomorphism $\M\times_{\O_E}E\cong{_K}M(G,X)$. If $\M$ is an integral model and $\gp$ a prime of $\O_E$, we write 
$\M_\gp$ and $\Mbar_\gp$ for stalk and fiber of $\M$ over $\gp$. Integral models exist and any two ones $\M_1$ and $\M_2$ are generically the same, 
i.e. there exists an integer $n$ and an isomorphism $\M_1\times\z[\frac{1}{n}]\cong\M_2\times\z[\frac{1}{n}]$ inducing the identity on the generic fiber, 
in particular one has $\M_{1,\gp}\cong\M_{2,\gp}$ and $\Mbar_{1,\gp}\cong\Mbar_{2,\gp}$ for all but finitely many primes. Once an integral model of 
${_K}M(G,X)$ is chosen we write $\T(g)$ for the Zariski closure of $T(g)$ in $\M\times_{\O_E}\M$, similarly we define $\T_\gp(g)$, and $\Tbar_\gp(g)$\\
Recall that a morphism $p:X\rightarrow Y$ between varieties is called generically finite if for every generic point $\eta$ of $X$ the residue field $k(\eta)$ is a 
finite extension of $k(p(\eta))$. The main result of this section is the following, $H_p(t)$ stands for the polynomial, that was denoted $H_{G\times\q_p,\mu}$ 
in section ~\ref{Hecke}.

\begin{defn}
\label{nvc}
Let ${_K}M(G,X)$ over $E$ be the canonical model coming from a Shimura datum $(G,X)$, and neat level structure $K$. Let $\M$ be an integral model. 
We say that ${_K}M(G,X)$ has no vertical components $(NVC)$ if for all but at most finitely many exceptional primes $\gp|p$ of $\O_E$, the two projection 
maps from $\Tbar_\gp(g)$ to $\Mbar_\gp$ are generically finite for all $g\in G(\q_p)$. By the previous remarks this does not depend on the choice of $\M$.
\end{defn}

\begin{thm}
\label{mI}
Let $G$ be a reductive $\q$-group which is an inner form of a split one. Let $(G,X)$ be a Shimura datum of Hodge type. 
Let $K$ be a neat level structure, and let $M={_K}M(G,X)$ be the corresponding canonical model over $\q$. Assume that 
it has $(NVC)$, and let $\M/\z$ be an integral model. Then one has for all but a finite number of rational primes $p$:
\begin{itemize}
\item
$\iota_p$ is a $\q$-linear map from $\H(G(\q_p)//K_p,\q)$ to the $\q$-vector space of correspondences $\q\otimes C(\M_p,\M_p)$.
\item
Denote by $\ibar_p$ the precomposition of $\iota_p$ with the specialization map $\sigma:\q\otimes C(\M_p,\M_p)\rightarrow\q\otimes C(\Mbar_p,\Mbar_p)$ 
as introduced in appendix ~\ref{speci}. Up to rational equivalence (in the sense of definition \ref{ccf}) $\ibar_p$ is a ring homomorphism, 
and in the $\q$-algebra $\q\otimes C_{rat}(\Mbar_p,\Mbar_p)$, the equation $\ibar_p(H_p)(\Gamma_p)=0$ holds, 
where $\Gamma_p\in C_{fin}(\Mbar_p,\Mbar_p)$ is the geometric Frobenius correspondence.
\end{itemize}
\end{thm}
\begin{proof}
As $(G,X)$ is a Hodge type datum, we can choose an integer $g$, and an imbedding $i:G\hookrightarrow GSp_{2g}\times\q$ 
carrying $X$ to $\gh_g^\pm$. By lemma ~\ref{subsh} there exists a level structure $H\subset GSp_{2g}(\ad)$ in order to 
obtain a closed immersion of canonical models $M\hookrightarrow A_{g,H}$. Let $Q$ be a finite set of primes, such that:
\begin{itemize}
\item
the base change to $\z[\prod_{p\in Q}p^{-1}]$ of $\M$ and of the Zariski closure of $M$ in $\A_{g,H}$ coincide,
\item
$H$ and $K$ have factorizations $H_Q\times\prod_{p\notin Q}GSp_{2g}(\z_p)$, and $K_Q\times\prod_{p\notin Q}K_p$, and $K_p$ is hyperspecial,
\item
when writing $\Mbar_p^\circ$ for the intersection of $\Mbar_p$ with $\Abar_{g,H,p}^\circ$, then $\Mbar_p^\circ$ is Zariski dense in $\Mbar_p$, for all $p\notin Q$,
\item
the set of exceptional primes referred to in definition ~\ref{nvc} is contained in $Q$,
\end{itemize}

Except for the density condition it is clear, that a finite set $Q$ of primes as above exists. The density at all but finitely many primes, finally is a 
result in ~\cite[theorem 1.1]{bu2}, see also ~\cite[p.577]{wedhorn}. The assertion $\iota_p(\H(G(\q_p)//K_p,\q))\subset\q\otimes C(\M_p,\M_p)$ 
is an immediate consequence of the following easy auxiliary lemma, we denote by $\M_p^\circ$, and $\T_p^\circ(g)$ the intersection 
of $\M_p$, and $\T_p(g)$ with $\A_{g,H,p}^\circ$ and $\A_{g,H,p}^\circ\times\A_{g,H,p}^\circ$.

\begin{lem}
\label{properfinite}
Let $p$ be a prime not in $Q$, let $g$ be an element in $G(\q_p)$. In this situation one has:
\begin{itemize}
\item[(i)]
the two projection maps $p_1,p_2:\T_p(g)\rightarrow\M_p$ are proper.
\item[(ii)]
the two projection maps $p_1,p_2:\T_p(g)^\circ\rightarrow\M_p^\circ$ are finite.
\end{itemize}
\end{lem}
\begin{proof}
For Siegel space $\A_{g,H}$ both facts are well known, (i) follows from the valuative criterion of properness in
~\cite[corollaire 7.3.10(ii)]{EGAII}, and (ii) follows from the quasi-finiteness of the maps in question, together with (i).\\
We deduce from this the corresponding results for the Zariski closure of the Shimura subvariety. The element $i(g)\in GSp_{2g}(\q_p)$ defines closed 
subschemes $\T_p(i(g))$ and $\T_p^\circ(i(g))$ of $\A_{g,H,p}\times\A_{g,H,p}$ and $\A_{g,H,p}^\circ\times\A_{g,H,p}^\circ$. Consider the diagram
$$\begin{CD}
\T_p(i(g))@>>>\A_{g,H,p}\times\A_{g,H,p}@>>>\A_{g,H,p}\\
@AAA@AAA@AAA\\
\T_p(g)@>>>\M_p\times\M_p@>>>\M_p
\end{CD}$$
It shows that the composition of $\T_p(g)\rightarrow\M_p$ with $\M_p\rightarrow\A_{g,H,p}$ is a proper map, 
which is enough to prove as all the involved schemes are separated. The same argument applies to part (ii).
\end{proof}

We write $A_i\in\H(G(\q_p)//K_p,\q)$ for the coefficients of the Hecke polynomial $H_p(t)=\sum_iA_it^i$. In order to prove the theorem it suffices to show 
that the $\q\otimes C(\Mbar_p,\Mbar_p)$-element given by $\sum_i\ibar_p(A_i)\Gamma_p^i$ vanishes for all $p\notin Q$ (N.B.: this is well-defined, since 
$\Gamma_p$ lies in $C_{fin}(\Mbar_p,\Mbar_p)$). To this end we appeal to lemma ~\ref{equality}. If one restricts Hecke correspondences 
to the ordinary locus they are finite by lemma ~\ref{properfinite}, so the assumptions of lemma ~\ref{equality} are fulfilled if one base changes to the 
algebraic closure $\f_p^{algcl}$. Pick any $\xbar\in \Mbar^\circ_p(\f_p^{algcl})$. By ~\cite[corollary 3.9]{noot} one can find a lifting to a $\O$-valued 
$CM$-point $\gx$, here $\O$ is a valuation ring of $\q^{algcl}$ corresponding to some prime $\gP|p$ of $\q^{algcl}$. The generic fiber of $\gx$ is of the 
form $x=[gK\times h]$, and thus gives rise to a rational torus $T\subset G$, and a cocharacter $\mu_{x,\gP}$. The coefficients of the Hecke polynomial are 
$K_p$-invariant $\q$-linear combinations of left $K_p$-cosets of $G(\q_p)$ that can be written as expressions of the form
$$A_i=\sum_jn_i^{(j)}g_i^{(j)}K_p\in\H(G(\q_p)//K_p,\q),$$
where the elements $g_i^{(j)}$ can be chosen in $P_{x,\gP}(\q_p)$. In a preliminary step we show that
$$\sigma(\sum_{i,j}n_i^{(j)}\gamma_\gP^i([gg_i^{(j)}K\times h]))=0$$
where $\gamma_\gP\in\Gal(\q^{algcl}/\q)$ is a Galois element that preserves $\O$ and induces the arithmetic Frobenius substitution on
$\O/\gP$. By using the reciprocity law it is equivalent to check that:
\begin{equation}
\label{prelim}
\sigma(\sum_{i,j}n_i^{(j)}[g\mu_{x,\gP}(p^{-i})g_i^{(j)}K\times h])=0
\end{equation}
The idea is to alter the group elements $g_i^{(j)}$ in a way which is justified by lemma ~\ref{moreconsequence}. Let $m_i^{(j)}$ be the image of $g_i^{(j)}$ in 
$M_{x,\gP}(\q_p)$. The lemma ~\ref{effect} gives the image of the coefficients $A_i$ of the Hecke polynomial under the map $\S\circ|_P$, they are equal to:
$$\S(A_i|_P)=\S(\sum_jn_i^{(j)}g_i^{(j)}(K_p\cap P_{x,\gP}(\q_p)))=\sum_jn_i^{(j)}m_i^{(j)}L_p$$
here $L_p$ is the group $U_{x,\gP}(\q_p)K_p/U_{x,\gP}(\q_p)$ which we regard as a subgroup of $M_{x,\gP}(\q_p)$ via the 
canonical splitting $P_{x,\gP}/U_{x,\gP}=M_{x,\gP}\hookrightarrow P_{x,\gP}$. Recall from lemma ~\ref{trivial} that the equation 
$H_p(\mu_{x,\gP}(p^{-1}))=0$ is valid in $\H(M_{x,\gP}(\q_p)//L_p,\q)$. By using lemma ~\ref{help} this is the same as saying that
$$0=\sum_{i,j}n_i^{(j)}\mu_{x,\gP}(p^{-i})m_i^{(j)}L_p,$$
holds in $\H(M_{x,\gP}(\q_p)/L_p,\q)$. Now fix a set of representatives $\gR\subset M_{x,\gP}(\q_p)$ for the left coset
decomposition $M_{x,\gP}(\q_p)=\bigcup_{r\in\gR}rL_p$. Consequently the previous equation implies that
$$0=\sum_{i,j}n_i^{(j)}r_i^{(j)}(K_p\cap P_{x,\gP}(\q_p))$$
holds in $\H(P_{x,\gP}(\q_p)/K_p\cap P_{x,\gP}(\q_p),\q)$, if the $r_i^{(j)}\in\gR$ denote the representatives for the $L_p$-cosets $\mu_{x,\gP}(p^{-i})m_i^{(j)}L_p$. It follows
$$0=\sum_{i,j}n_i^{(j)}[gr_i^{(j)}K\times h]$$
as elements in $Z_0(M\times\q^{algcl})$. As $r_i^{(j)}$ and $\mu_{x,\gP}(p^{-i})g_i^{(j)}$ have like image in $M_{x,\gP}(\q_p)/L_p$ there exist $u_i^{(j)}\in U_{x,\gP}(\q_p)$ with
$$r_i^{(j)}K_p=u_i^{(j)}\mu_{x,\gP}(p^{-i})g_i^{(j)}K_p.$$
Upon combining with
$$\overline{[gr_i^{(j)}K\times h]}=\overline{[g\mu_{x,\gP}(p^{-i})g_i^{(j)}\times h]}$$
one gets ~\eqref{prelim}.\\
Now the Hecke correspondences $\iota_p(A_i)|_{\M_p^\circ\times\M_p^\circ}$ are finite by lemma ~\ref{properfinite}, so lemma ~\ref{evaluate} can be used to compute
$$0=\sigma(\sum_ix^{\gamma_\gP^i}.\iota_p(A_i))=\sum_i\sigma(x^{\gamma_\gP^i}).\ibar_p(A_i)$$
And finally the compatibility of specialization with base change together with lemma ~\ref{frob} tells us that
$$\sum_i\sigma(x^{\gamma_\gP^i}).\ibar_p(A_i)=\sum_i\xbar.\Gamma_p^i.\ibar_p(A_i)$$
this concludes the proof.
\end{proof}

\section{Examples of canonical models with $(NVC)$}
\label{example}
We describe a particular Shimura datum: Let $D$ be a definite $\q$-quaternion algebra, let $\O_D$ be a maximal order and write $\bar{\cdot}$ for the 
standard involution. Let $\bigoplus_{i=1}^2\O_De_i\oplus\O_Df_i$ be the free $\O_D$-module of rank $4$, endow it with a $\z$-valued form by: 
$(\sum_{i=1}^2a_ie_i+b_if_i,\sum_{i=1}^2a_i^\prime e_i+b_i^\prime f_i)=\tr_{D/\q}(\sum_{i=1}^2a_i{\bar b}_i^\prime-b_i{\bar a}_i^\prime)$. Let $G_D$ 
be the group $\z$-scheme of $\O_D$-linear similitudes of this skew-Hermitian $\O_D$-module. Let $G_D^\circ$ be the connected component. A conjugacy 
class $X$ of maps from $\c^\times$ to $G_D^\circ(\r)$ is provided by making $h(a+\sqrt{-1}b)$ send $e_i$ to $ae_i-bf_i$ and $f_i$ to $be_i+af_i$. The 
reductive group $G_D\times\q$ is an inner form of the split $GO(8)$, the pair $(G_D^\circ\times\q,X)$ is a Shimura datum, one checks that $\dim X=6$.

\begin{thm}
Let $(G,Y)$ be a Shimura datum. Assume that $\dim Y=5$ or $\dim Y=6$. Assume also that there exists an imbedding $G\hookrightarrow G_D^\circ\times\q$ 
taking $Y$ to the conjugacy class $X$. Then ${_K}M(G,Y)$ has $(NVC)$ for any neat compact open subgroup $K\subset G(\ad)$.
\end{thm}
\begin{proof}
If $K$ and $i$ are given one finds a level structure $H\subset G_D^\circ(\ad)$ giving rise to a closed immersion 
$M={_K}M(G,Y)\hookrightarrow A={_H}M(G_D^\circ\times\q,X)$ Let $K=K_Q\times\prod_{p\notin Q}K_p$, and $H=H_Q\times\prod_{p\notin Q}G_D^\circ(\z_p)$, 
be factorizations with $Q$ a finite set such that $K_p$ and $G_D^\circ(\z_p)$ are hyperspecial for $p\notin Q$. The canonical model $A$ has a 
moduli interpretation of PEL type, see ~\cite[chapter 5]{kottwitz} for details, it gives rise to an integral model $\A$. Let $\M$ be the Zariski closure of 
$M$ in $\A$. We set $\Sbar_p$ for the supersingular locus of $\Abar_p$, for some $p\notin Q$. To an element $g\in G(\q_p)$ we consider the diagram
$$\begin{CD}
\T_p(i(g))@>p_1>>\A_p\\
@AAA@AAA\\
\T_p(g)@>p_1>>\M_p
\end{CD}$$
as we did in the proof of lemma ~\ref{properfinite}. Using methods of Katsura and Oort one can check that
\begin{itemize}
\item
$\dim\Sbar_p=2$,
\item
over any point $\eta\in\Abar_p-\Sbar_p$ the fibers of the projection map $p_1$ are finite,
\end{itemize}
see ~\cite[theorem 2.4]{bu3} and ~\cite[corollary 2.7]{bu3} for full details. We can now show that $p:\Tbar_p(g)\rightarrow\Mbar_p$ is generically finite. 
Let $\Cbar$ be an irreducible component of $\Tbar_p(g)$. Let $\eta$ be the generic point of $\Cbar$. If $p_1$ maps $\eta$ to a point not in $\Sbar_p$ we 
are done as the fibers over $\Abar_p-\Sbar_p$ are finite. If $p_1$ does map $\eta$ into $\Sbar_p$, one must have $\Cbar\subset\Sbar_p\times\Sbar_p$, as 
supersingularity is an isogeny invariant. By flatness of $\T_p(g)$ we have $\dim\Cbar=5$ or $\dim\Cbar=6$, contradicting $\dim\Sbar_p\times\Sbar_p=4$.
\end{proof}

\begin{rem}
Using Clifford algebras in the way Satake does in ~\cite{satake2}, one can indeed show that $5$-dimensional Shimura data $(G,Y)$ as in the theorem exist. Such a $G$ 
is the spinor similitude group of a certain $7$-dimensional quadratic $\q$-space of signature $(2,5)$, moreover, by an easy application of Meyer's theorem the spinor 
similitude group of every $\q$-form of the quadratic $\r$-space of signature $(2,5)$ can be imbedded into $G_D$, for some $D$ that depends on the quadratic $\q$-space.
\end{rem}

\begin{appendix}
\section{The algebra of correspondences}
\label{algebra}
This appendix is devoted to some elementary properties of the algebra of correspondences. Recall some concepts from the theory of algebraic cycles: To every 
algebraic variety $X$ over a field one can attach abelian groups $Z_n(X)$ and $A_n(X)$. The group of $n$-cycles $Z_n(X)$ is defined as the free abelian group 
generated by all irreducible $n$-dimensional subvarieties of $X$. The divisors of functions on all $n+1$-dimensional subvarieties generate the subgroup of 
cycles that are rationally equivalent to $0$. The quotient of $Z_n(X)$ by this subgroup is $A_n(X)$. In ~\cite[chapter 20]{fulton} it is explained how to extend 
these concepts to any noetherian universally catenary scheme $\X$. A cycle is still a formal sum of closed irreducible subschemes, but the Krull dimension of 
the summands is not useful to define the dimension of the cycle. One rather fixes a morphism $p:\X\rightarrow S$ to an auxiliary base scheme $S$ and defines 
$Z_n(\X/S)$, the group of $n$-cycles of $\X$ relative to $S$, to be the free abelian group generated by all closed irreducible $\C\hookrightarrow\X$ which satisfy
$$\trdeg(k(\eta)/k(p(\eta)))-\codim_Sp(\eta)=n$$
if $\eta$ is the generic point of $\C$. The expression on the left is called the relative dimension of $\C$ and denoted $\dim_S\C$, it may very 
well be a negative number, finally we denote the support of some $\C\in Z_n(\X/S)$ by $|\C|\subset\X$. In the following we will stick to the special 
case $S=\Spec R$, with $R$ either a discrete valuation domain or a field. Again by ~\cite[chapter 20]{fulton} one knows that most of the basic 
facts of intersection theory, for example those proved in ~\cite[chapters 1-6]{fulton} remain valid in this context. This means in particular that:
\begin{itemize}
\item[(I)]
if all irreducible components of $\X$ have the same relative dimension $n$ over $S$ then $\X$ has a fundamental cycle $[\X]\in Z_n(\X/S)$
\item[(II)]
there is a notion of rational equivalence of cycles of relative dimension $n$ over $S$, which is used to define the quotient $A_n(\X/S)$ of $Z_n(\X/S)$, 
and if no irreducible component of $\X$ has relative dimension strictly larger than $n$, the canonical map from $Z_n(\X/S)$ to $A_n(\X/S)$ is a bijection.
\item[(III)]
for every proper morphism $f:\X\rightarrow\Y$ of schemes over $S$, there is a proper push-forward 
$f_*:Z_n(\X/S)\rightarrow Z_n(\Y/S)$ which passes to rational equivalence to define a map $f_*:A_n(\X/S)\rightarrow A_n(\Y/S)$
\item[(IV)]
one can define an exterior product $\times_S:Z_n(\X/S)\times Z_m(\Y/S)\rightarrow Z_{n+m}(\X\times_S\Y)$ which also passes to a product of cycles up to rational equivalence
\item[(V)]
for every regular embedding $i:\X\hookrightarrow\Y$ of schemes over $S$ there is a refined Gysin map 
$i^!:A_n(\Y'/S)\rightarrow A_{n-d}(\X'/S)$, with $\Y'$ any scheme over $\Y$, $\X'=\X\times_\Y\Y'$ and $d$ the codimension of the embedding $i$
\item[(VI)]
for any open subscheme $U\subset\X$ one can restrict cycles from $\X$ to $U$ in order to obtain maps $|_U:Z_n(\X/S)\rightarrow Z_n(U/S)$ 
and $|_U:A_n(\X/S)\rightarrow A_n(U/S)$, these are special cases of flat pull-back maps (which we will not use).
\end{itemize}

It will be handy to have a special notation for certain sets of correspondences:

\begin{defn}
\label{ccf}
Let $X$ and $Y$ be smooth $d$-dimensional algebraic varieties over $k$.
\begin{itemize}
\item
A correspondence $C$ from $X$ to $Y$ is a $d$-cycle
$$C=\sum_in_i[V_i]$$
in $X\times_kY$ which is supported on some closed $d$-dimensional subvariety $|C|\subset X\times_kY$, such that both projection maps 
$p_X:|C|\rightarrow X$ and $p_Y:|C|\rightarrow Y$ are proper. We write $C\sim0$ if $|C|$ is contained in some closed subvariety $S\subset X\times_kY$, 
such that both projection maps $p_X:S\rightarrow X$ and $p_Y:S\rightarrow Y$ are proper, and $C$ vanishes in $A_d(S)$. Let $C(X,Y)$ 
be the abelian group of all correspondences from $X$ to $Y$, and let $C_{rat}(X,Y)$ be its quotient by the
equivalence relation $\sim$. 
\item
An element $C\in C(X,Y)$ is called a finite correspondence if the projection maps $p_X:|C|\rightarrow X$ and 
$p_Y:|C|\rightarrow Y$ are finite. Let $C_{fin}(X,Y)$ denote the subgroup of finite correspondences from $X$ to $Y$.
\item
Let $R$ be a discrete valuation ring and let $\X$ and $\Y$ be schemes over $R$ which are smooth of relative dimension $d$, and let $X$ and $Y$ be 
their generic fibers. Consider a cycle $C\in Z_d(X\times_kY)$ and let $|\C|$ be the Zariski closure of $|C|$ in $\X\times_R\Y$. We call $C$ a correspondence 
from $\X$ to $\Y$ if the projection maps from $|\C|$ to the factors $\X$ and $\Y$ are proper. Let $C(\X,\Y)$ denote the abelian group of all correspondences 
from $\X$ to $\Y$. Furthermore, we let $C_{fin}(\X,\Y)$ be the subgroup of correspondences the closure of whose supports are finite over both $\X$ and $\Y$, 
while $C_{rat}(\X,\Y)$ denotes the quotient of $C(\X,\Y)$ by the subgroup of $C$ that vanish in some $A_d(\S/R)$ for some closed $\S\subset\X\times_R\Y$ which is 
proper over $\X$ and $\Y$, and contains $|\C|$.
\end{itemize}
\end{defn}

Note the natural maps $C_{fin}(X,Y)\hookrightarrow C(X,Y)\twoheadrightarrow C_{rat}(X,Y)$ and 
$C_{fin}(\X,\Y)\hookrightarrow C(\X,\Y)\twoheadrightarrow C_{rat}(\X,\Y)$. We sketch how elements in $C_{rat}$ can be multiplied: Consider smooth $d$-dimensional algebraic 
varieties $X$, $Y$ and $Z$ over $k$, and reduced closed subvarieties $S\hookrightarrow X\times_kY$ and $T\hookrightarrow Y\times_kZ$. If each of $S$ and $T$ is 
proper over the factors, than so is the natural map $p:S\times_YT\rightarrow X\times_kZ$, in which case we let $S*T$ be its (mere set-theoretic) image. Observe that 
$S*T$ is a Zariski closed subset of $X\times_kZ$, of which the reduced induced subscheme structure is proper over $X$ and $Z$. Given any $C\in A_d(S)$ and 
$D\in A_d(T)$, we consider the Cartesian square
$$\begin{CD}
S\times_YT@>>>X\times_kY\times_kZ\\
@VVV@ViVV\\
S\times_kT@>>>X\times_kY\times_kY\times_kZ
\end{CD}$$
with associated refined Gysin map (cf. ~\cite[chapter 6]{fulton})
$$i^!:A_{2d}(S\times_kT)\rightarrow A_d(S\times_YT),$$
and push forward:
$$p_*:A_d(S\times_YT)\rightarrow A_d(S*T)$$
We let the product $C.D$ be the element of $C_{rat}(X,Z)$ which is defined by: 
$$p_*(i^!(C\times D))\in A_d(S*T),$$ 
and similarly we obtain the bilinear map $$C_{rat}(\X,\Y)\times C_{rat}(\Y,\Z)\rightarrow C_{rat}(\X,\Z)$$ for smooth $R$-schemes $\X$, $\Y$ and $\Z$.

\begin{rem}
\label{blablabla}
Consider correspondences $C\in C(X,Y)$ and $D\in C(Y,Z)$. If at least one of them is finite, then the above procedure can be used to 
define a product $C.D\in C(X,Z)$. This is due to $\dim(|C|*|D|)\leq d$ in this case, so that $A_d(|C|*|D|)=Z_d(|C|*|D|)$. Similarly finite 
correspondences $C\in C_{fin}(X,Y)$, and $D\in C_{fin}(Y,Z)$ give rise to a product $C.D\in C_{fin}(X,Z)$, as $|C.D|\subset |C|*|D|$.
\end{rem}

Let us check the associativity:

\begin{lem}
\label{Assoziativ}
Let $X$, $Y$, $Z$ and $W$ be smooth $d$-dimensional varieties over a field $k$. Let 
$C\in C_{rat}(X,Y)$, $D\in C_{rat}(Y,Z)$ and $E\in C_{rat}(Z,W)$. Then $(C.D).E=C.(D.E)$
\end{lem}
\begin{proof}
By slight abuse of notation we denote suitable representatives of $C$, $D$ and $E$ in $C(X,Y)$, $C(Y,Z)$ and $C(Z,W)$ by the same 
letter, and we write $|C|$, $|D|$ and $|E|$ for their supports. Let $B\in A_d(|C|\times_Y|D|)$ be the Gysin pull back of $C\times D$ along 
$X\times_kY\times_kZ\rightarrow X\times_kY\times_kY\times_kZ$. Consider the following commutative diagram in which the squares are Cartesian.
$$\begin{CD}
|C|\times_Y|D|\times_Z|E|@>>>|C|\times_Y|D|\times_k|E|\\
@VqVV@VpVV\\
(|C|*|D|)\times_Z|E|@>>>(|C|*|D|)\times_k|E|\\
@VVV@VVV\\
X\times_kZ\times_kW@>i>>X\times_kZ\times_kZ\times_kW\\
@VVV\\
X\times_kW
\end{CD}$$
Put $A=i^!(B\times E)\in A_d(|C|\times_Y|D|\times_Z|E|)$. By ~\cite[theorem 6.2.(a)]{fulton} one has $q_*A=i^!p_*(B\times E)$. This shows 
that $(C.D).E=p_{XW*}A$. Here $p_{XW}$ is the obvious map from $|C|\times_Y|D|\times_Z|E|$ to $|C|*|D|*|E|\subset X\times_kW$. On the 
other hand, if $j$ denotes the map from $X\times_kY\times_kZ\times_kW$ to $X\times_kY\times_kY\times_kZ\times_kZ\times_kW$ one gets from the diagram
$$\begin{CD}
|C|\times_Y|D|\times_Z|E|@>>>X\times_kY\times_kZ\times_kW\\
@VVV@VVV\\
|C|\times_Y|D|\times_k|E|@>>>X\times_kY\times_kZ\times_kZ\times_kW\\
@VVV@VVV\\
|C|\times_k|D|\times_k|E|@>>>X\times_kY\times_kY\times_kZ\times_kZ\times_kW
\end{CD}$$
that $A=j^!(C\times D\times E)$. A variant of the above argument shows that one has also $C.(D.E)=p_{XW*}j^!(C\times D\times E)$ 
which proves the lemma.
\end{proof}

\begin{defn}
Let $X$ and $Y$ be smooth algebraic varieties over $k$ of dimension $d$. Let $P\in Z_n(X)$, and $C\in C_{fin}(X,Y)$. 
Let $|P|\subset X$, and $|C|\subset X\times_kY$ be their supports. Consider the Cartesian square
$$\begin{CD}
|P|\times_X|C|@>>>X\times_kY\\
@VVV@ViVV\\
|P|\times_k|C|@>>>X\times_kX\times_kY
\end{CD}$$
with associated refined Gysin map:
$$i^!:A_{n+d}(|P|\times_k|C|)\rightarrow A_n(|P|\times_X|C|)=Z_n(|P|\times_X|C|),$$
here the equation $A_n(\dots)=Z_n(\dots)$ is due to $\dim(|P|\times_X|C|)=n$ by finiteness of $C$. Consider further the natural map 
$p:|P|\times_X|C|\rightarrow Y$, it is finite as $|C|$ is finite over $Y$. Let $|P|*|C|\subset Y$ be the image of $p$, with associated push forward:
$$p_*:Z_n(|P|\times_X|C|)\rightarrow Z_n(|P|*|C|)$$
Then the element $P.C=p_*(i^!(P\times C))\in Z_n(|P|*|C|)$ constitutes an element in $Z_n(Y)$ which is called the product of $P$ with $C$.
\end{defn}

Finite correspondences act on $n$-cycles: If $P$ is a $n$-cycle of $X$ and $C\in C_{fin}(X,Y)$ and $D\in C_{fin}(Y,Z)$, then one has $(P.C).D=P.(C.D)$. 
Occasionally it is useful to consider cycles which are only defined over some extension of the ground field: Let $l/k$ be a finite field extension and define a base change map
$$\times_kl:Z_n(X)\hookrightarrow Z_n(X\times_kl)$$
by sending a cycle $C$ to the exterior product $C\times[\Spec l]$. It is a fact that this operation is compatible with the constructions in this subsection i.e. $\times_kl$ 
maps $C(X,Y)$ into $C(X\times_kl,Y\times_kl)$ and commutes with products and preserves the subgroups of correspondences that are finite or rationally equivalent to 
$0$. If one fixes an algebraic closure $k^{algcl}$ of $k$, then one can carry this over to the limit over all finite extensions to thus arrive at an inclusion 
$Z_n(X)\hookrightarrow Z_n(X\times_kk^{algcl})$ and similarly $C(X,Y)\hookrightarrow C(X\times_kk^{algcl},Y\times_kk^{algcl})$. On these groups there is an action 
of the automorphism group of $k^{algcl}/k$.\\
Let $X$ be a variety defined over $\f_q$, and assume it is equidimensional of dimension $d$. We define a particular 
$d$-cycle $\Gamma_q\in Z_d(X\times_{\f_q}X)$ as the graph of the geometric Frobenius map $x\mapsto x^q$. Indeed one 
has $\Gamma_q\in C_{fin}(X,X)$ because the Frobenius morphism is a finite map. The proof of the following lemma is clear:

\begin{lem}
\label{frob}
Let $X$ be a $d$-dimensional variety over $\f_q$. Let $\gamma_q\in\Gal(\f_q^{algcl}/\f_q)$ be the 
arithmetic Frobenius automorphism. Then one has for all $x\in Z_0(X\times_{\f_q}\f_q^{algcl})$:
$$x.\Gamma_q=x^{\gamma_q}$$
\end{lem}

Let us us say that some $C\in C(X,Y)$ is generically finite if both projection maps $p_X:|C|\rightarrow X$ and $p_Y:|C|\rightarrow Y$ have that property.

\begin{lem}
\label{equality}
Let $X$ and $Y$ be smooth $d$-dimensional varieties over $k$. Assume that $k$ is algebraically closed, and let 
$C$ be a generically finite correspondence from $X$ to $Y$. Assume that there are given open dense 
subvarieties $X^\circ\subset X$ and $Y^\circ\subset Y$, such that the restriction $C^\circ=C|_{X^\circ\times_kY^\circ}$ 
is an element of $C_{fin}(X^\circ,Y^\circ)$. Assume that $x.C^\circ$ vanishes for all closed points $x\in X^\circ$. Then $C$ vanishes.
\end{lem}
\begin{proof}
It is clear that $C$ vanishes if and only if $C^\circ$ does, so we may assume without loss of generality that $X^\circ=X$ and $Y^\circ=Y$. 
Assume $C$ was not zero and write it as $C=\sum n_i[V_i]$ with nonzero $n_i$ and mutually different $V_i$. Write $p_X$ for the projection 
onto the factor $X$ and consider the set $F=p_X(V_1\cap(\bigcup_{i\neq1}V_i))$. It is a Zariski closed subset of $X$ of dimension strictly 
smaller than $d$. We may pick a closed point $x\in X-F$, which leads to $x.C=n_1x.[V_1]+\sum_{i\neq1}n_ix.[V_i]$. The supports of the 
$0$-cycles $x.[V_1]$ and $\sum_{i\neq1}n_ix.[V_i]$ are disjoint, but we certainly have $x.[V_1]\neq0$. The contradiction $x.C\neq0$ follows.
\end{proof}

\section{Specialization of cycles}
\label{speci}

So far we have only considered varieties over fields, now let $\X$ be a scheme over a discrete valuation ring $R$ with generic fiber $X$ and special fiber $\Xbar$. We need 
to specialize cycles in $X$ to cycles in $\Xbar$. If $\X$ is flat over $R$ the special fiber becomes a Cartier divisor of $\X$ and we want to think of the specialization process 
as a special case of the operation to intersect a cycle with a Cartier divisor. In the equal characteristic case this operation is introduced in ~\cite[definition/remark 2.3]{fulton}, 
but the supplements in ~\cite[chapter 20]{fulton} show again that it also exists in the slightly more general setting of a scheme $\X$ over an arbitrary -possibly mixed 
characteristic- discrete valuation ring. We let $D$ be a Cartier divisor of $\X$ write $|D|$ for its support and let $U$ be the complement of $|D|$ in $\X$. Define a map
$$Z_n(U/R)\rightarrow Z_{n-1}(|D|/R);\alpha\mapsto D.\alpha$$
by the following considerations: $D$ determines a line bundle $\O(D)$. On $U$ this line bundle is canonically trivialized i.e. we have a on $U$ nowhere 
vanishing section $s\in\Gamma(U,\O(D))$. Without loss of generality we may assume that our $n$-cycle $\alpha$ is of the form $[C]$ with $C\subset U$ 
being closed irreducible of relative dimension $n$ over $R$. Let $\C$ be the Zariski closure of $C$ in $\X$. By restricting $\O(D)$ and $s$ one obtains a line 
bundle $\L$ on $\C$ which is trivialized on the set $U\cap\C$. The pair $\L$ together with its given trivialization over $U\cap\C$ constitutes a so-called pseudo-divisor. 
It therefore has an associated Weil divisor (cf. ~\cite[lemma 2.2(a)]{fulton}) which one can write as a sum $W=\sum n_i[W_i]$, in which every $W_i$ is of codimension 
one in $\C$, hence of relative dimension $n-1$ over $R$. One defines $D.\alpha$ to be $W$. In the special case $D=\Xbar$ one just has to note that $U$ 
is $X$, that both $X$ and $\Xbar$ are varieties defined over fields, namely the fraction and the residue field of $R$ and that one has natural equalities
$$Z_n(X/R)=Z_n(X)$$
and
$$Z_{n-1}(\Xbar/R)=Z_n(\Xbar)$$
so that the intersection with $\Xbar$ defines a map
$$\sigma:Z_n(X)\rightarrow Z_n(\Xbar)$$
which is called the specialization map. Note that in this special case the occurring divisors and pseudo-divisors are principal, namely the ones defined by a 
uniformizer of $R$, say $\pi$. Thus the computation of $\sigma(C)$ boils down to taking the Weil divisor of $\pi$ on $\C$ which is the same as the fundamental 
cycle of the closed subscheme which it defines, i.e. we have that $\sigma(C)=W=[\Cbar]$, if $\Cbar=\C\times_RR/\pi$ denotes the special fiber of $\C$.\\
We will need to check the compatibility of $\sigma$ with the various product operations defined in the last subsection. For doing this yet another description of 
$\sigma$, namely one in terms of the refined Gysin map will prove to be useful. Consider a cycle $C\in Z_n(X)$, let $\C$ be the same cycle regarded as an element 
in $Z_n(\X/R)$. Let $|C|\subset X$ be the support of $C$, let $|\C|$ be the support of $\C$, $|\C|$ is the Zariski closure of $|C|$ in $\X$, let finally $|\Cbar|$ be the 
special fiber of $|\C|$. Corresponding to the regular embedding $i:\Spec k\hookrightarrow\Spec R$, which has codimension one, there is a refined Gysin map
$$i^!:A_n(|\C|/R)\rightarrow A_{n-1}(|\Cbar|/R).$$
However, because $|\C|$, and $|\Cbar|$ have relative dimensions $n$ and $n-1$ over $R$, we have $A_n(|\C|/R)=Z_n(|\C|/R)$ and
$A_{n-1}(|\Cbar|/R)=Z_{n-1}(|\Cbar|/R)$, so that $i^!(\C)$ is a honest relative $n-1$-cycle in $|\Cbar|$ over $R$, now use the obvious
inclusions
$$Z_n(|\C|/R)=Z_n(|C|)\subset Z_n(X)$$
$$Z_{n-1}(|\Cbar|/R)=Z_n(|\Cbar|)\subset Z_n(\Xbar)$$
to recover $\sigma(C)$ from $i^!(\C)$. Note also that the specialization map commutes with base change to a bigger discrete valuation 
ring $R'$. One could deduce this from ~\cite[proposition 2.3(d)]{fulton}, but let us argue directly: Consider $\X'=\X\times_RR'$ and 
write $X'$ and $\Xbar'$ for the generic and special fiber of $\X'$. We want to settle the commutativity of the diagram
$$\begin{CD}
Z_n(X)@>\sigma>>Z_n(\Xbar)\\
@V\times_KK'VV@V\times_kk'VV\\
Z_n(X')@>\sigma>>Z_n(\Xbar')
\end{CD}$$
in which $K'/K$ and $k'/k$ denote the fraction and residue field extensions corresponding to $R'/R$. The fundamental cycle of any closed irreducible $C\subset X$ with Zariski 
closure $\C\subset\X$ specializes to $[\C\times_Rk]$ and base changes to $[C\times_KK']$. It is then evident that $\C\times_RR'$ is the Zariski closure of $C\times_KK'$ so that this cycle specializes to $[(\C\times_RR')\times_{R'}k']=[\C\times_Rk']=\sigma([C])\times_kk'$. We will frequently consider $K'=K^{algcl}$ and want to choose a valuation ring 
$R'$ of $K'$ which dominates $R$. As $R'$ may not be discrete one cannot apply the above recipe directly, but observe that every element of $Z_n(X\times_KK^{algcl})$ 
is actually defined over some finite extension of $K$ which will cut out a subring of $R'$ which is a discrete valuation ring. This consideration sets up a map
$$\sigma:Z_n(X\times_KK^{algcl})\rightarrow Z_n(\Xbar\times_kk^{algcl})$$
which we will still call the specialization map. Here $k^{algcl}$ is the residue field of $R'$.\\

At last notice, that we also may specialize correspondences between smooth $d$-dimensional 
$R$-schemes $\X$ and $\Y$ to their respective special fibers $\Xbar$ and $\Ybar$, to yield linear maps:
$$\sigma_{fin}:C_{fin}(\X,\Y)\rightarrow C_{fin}(\Xbar,\Ybar)$$
and
$$\sigma_{rat}:C_{rat}(\X,\Y)\rightarrow C_{rat}(\Xbar,\Ybar),$$
of which we want to check the multiplicativity:

\begin{lem}
\label{evaluate}
Let $R$ be a discrete valuation ring with fraction field $K$ and residue field $k$.
\begin{itemize}
\item
Let $\X$, $\Y$ and $\Z$ be smooth of relative dimension $d$ over $R$. Let $X$, $Y$, $Z$ be the generic and $\Xbar$, $\Ybar$ and $\Zbar$ 
be the special fibers. Let $C\in C_{rat}(\X,\Y)$ and $D\in C_{rat}(\Y,\Z)$. Then one has $$\sigma_{rat}(C.D)=\sigma_{rat}(C).\sigma_{rat}(D)$$
\item
Let $\X,\Y/R$ be as before. Consider a $n$-cycle $P\in Z_n(X)$ and a finite correspondence $C\in C_{fin}(\X,\Y)$. Then $\sigma(P).\sigma_{fin}(C)$ and $P.C$ are well-defined and
$$\sigma(P.C)=\sigma(P).\sigma_{fin}(C)$$
holds in $Z_n(\Ybar)$.
\end{itemize}
\end{lem}
\begin{proof}
Let us choose representatives $\C\in Z_d(\S/R)$ and $\D\in Z_d(\T/R)$ for suitable closed $\S\subset\X\times_R\Y$ and $\T\subset\Y\times_R\Z$, which are 
proper over each of the factors $\X$, $\Y$ and $\Z$. Let $\C\times\D$ be their exterior product in $Z_{2d}(\S\times_R\T/R)$. Consider its Gysin pull-back along 
$\X\times_R\Y\times_R\Z\rightarrow\X\times_R\Y\times_R\Y\times_R\Z$ and call it $\B$. Let $\Bbar\in A_d(\Sbar\times_{\Ybar}\Tbar)$ be the Gysin pull-back of 
$\Cbar\times\Dbar$ along $\Xbar\times_k\Ybar\times_k\Zbar\rightarrow\Xbar\times_k\Ybar\times_k\Ybar\times_k\Zbar$, where $\Sbar:=\S\times_Rk$ and 
$\Tbar:=\T\times_Rk$. Consider the following commutative diagram:
$$\begin{CD}
\Spec k@>i>>\Spec R\\
@AAA@AAA\\
\Sbar\times_k\Tbar@>>>\S\times_R\T@>>>\X\times_R\Y\times_R\Y\times_R\Z\\
@AAA@AAA@AAA\\
\Sbar\times_{\Ybar}\Tbar@>>>\S\times_{\Y}\T@>>>\X\times_R\Y\times_R\Z
\end{CD}$$
By ~\cite[Theorem 6.4.]{fulton} it allows to conclude that $i^!(\B)$ is the same as $\Bbar$. This does not yet 
prove that $\sigma_{rat}(C.D)=\sigma_{rat}(C).\sigma_{rat}(D)$, but the diagram
$$\begin{CD}
\Sbar\times_{\Ybar}\Tbar@>>>\S\times_{\Y}\T\\
@Vp_{\Xbar\Zbar}VV@Vp_{\X\Z}VV\\
\Xbar\times_k\Zbar@>>>\X\times_R\Z\\
@VVV@VVV\\
\Spec k@>i>>\Spec R
\end{CD}$$
shows $\sigma_{rat}(C).\sigma_{rat}(D)=p_{\Xbar\Zbar*}(\Bbar)=p_{\Xbar\Zbar*}(i^!\B)=i^!p_{\X\Z*}(\B)=\sigma_{rat}(C.D)$ 
according to ~\cite[theorem 6.2.(a)]{fulton}.\\ 
The proof of the second part of the lemma is exactly analogous.
\end{proof}

\end{appendix}


\begin{thebibliography}{10}

\bibitem{blasius}
D. Blasius, J.D. Rogawski, {\em Zeta Functions of Shimura Varietes} Proc. Symp. Pure Math. 55

\bibitem{bu1}
O. B\"ultel, {\em On the mod $\gP$-reduction of ordinary $CM$-points} Oxford D.Phil. thesis, trinity term 1997

\bibitem{bu2}
O. B\"ultel, {\em Density of the Ordinary Locus} Bull. London Math. Soc. 33(2001)

\bibitem{bu3}
O. B\"ultel, {\em On the Supersingular Loci of Quaternionic Siegel Spaces} to appear in Cubo: A mathematical journal

\bibitem{cartier}
P. Cartier, {\em Representations of $\gp$-adic groups: A survey} Proc. Symp. Pure Math. 33 (1979), part 1, p.111-155

\bibitem{casselman}
W. Casselman, {\em The Hasse-Weil $\zeta$-function of some moduli varieties of dimension greater than one} Proc. Symp. Pure Math. 33 (1979), part 2, p.141-163

\bibitem{bourbaki}
P. Deligne, {\em Travaux de Shimura}, SLN 244

\bibitem{corvallis}
P. Deligne, {\em Vari\'et\'es de Shimura: Interpr\'etation Modulaire, et Techniques de Construction de Mod\`eles Canonique} Proc. Symp. Pure Math. 33 (1979), part 2, p.247-290

\bibitem{deligne}
P. Deligne {\em Vari\'et\'es ab\'eliennes ordinaires sur un corps fini}
Inv. Math. 8(1969), p.238-243

\bibitem{faltings}
G. Faltings, C-L. Chai, {\em Degeneration of Abelian Varieties} Ergebnisse der Mathematik und ihrer Grenzgebiete Band 22

\bibitem{mumford}
J. Fogarty, F. Kirwan, D. Mumford, {\em Geometric Invariant Theory} Ergebnisse der Mathematik und ihrer Grenzgebiete Band 34

\bibitem{fulton}
W. Fulton, {\em Intersection Theory} Ergebnisse der Mathematik und ihrer Grenzgebiete Band 2

\bibitem{EGAII}
A. Grothendieck, J. Dieudonn\'e, {\em \'Etude globale \'el\'ementaire de quelques classes de morphismes (EGA II)}

\bibitem{kottwitz}
R. Kottwitz, {\em Points on some Shimura varieties over finite fields} J. Amer. Math. Soc. 5 (1992) pp.373-444

\bibitem{noot}
R. Noot, {\em Models of Shimura varieties in mixed characteristic} J. Algebraic Geometry 5 (1996) p.187-207

\bibitem{satake1}
I. Satake, {\em Theory of spherical functions on reductive algebraic groups over $\gp$-adic fields} Inst. Hautes Etudes Sci. Publ. Math. 18 (1963) p.1-69

\bibitem{satake2}
I. Satake, {\em Clifford Algebras and Families of Abelian Varieties} Nagoya math. Journal 27-2 (1966) p.435-446 and 31 (1968) p.295-296

\bibitem{shimura}
G. Shimura, {\em Introduction to the arithmetic theory of automorphic functions}, Princeton 1971

\bibitem{serretate}
J-P. Serre, J. Tate, {\em Good Reduction of Abelian Varieties} Ann. of Math. 88 (1968) p.492-517

\bibitem{hondatate}
J. Tate, {\em Classes d'isog\'enie des vari\'et\'es ab\'eliennes sur un corps fini} S\'eminaire Bourbaki Expos\'es 347-363, SLN 179, pp.95-110

\bibitem{wedhorn}
T. Wedhorn, {\em Ordinariness in good reductions of Shimura varieties of PEL-type} Ann. scient. \'Ec. Norm. Sup. quatri\`eme s\'erie 32 (1999) pp.575-618

\end{thebibliography}
\end{document}